\documentclass{amsart}
\pdfoutput=1

\usepackage[dvipsnames]{xcolor}
\usepackage[unicode,
  colorlinks=true,
  linktocpage=true,
  citecolor=Emerald,
  linkcolor=Fuchsia,
  urlcolor=Fuchsia,
  pdftitle={Dimension and partial groups},
  pdfauthor={Philip Hackney and Rémi Molinier},
  pdfkeywords={dimension, partial group, groupoid, symmetric simplicial set}
  ]{hyperref}

\usepackage{amssymb}
\usepackage[capitalise]{cleveref}
\usepackage{comment}
\usepackage{mathtools}
\usepackage{microtype}
\newcommand{\sset}{\mathsf{sSet}}
\newcommand{\sym}{\mathsf{Sym}}
\newcommand{\fun}{\mathsf{Fun}}
\newcommand{\set}{\mathsf{Set}}
\newcommand{\op}{\textup{op}}
\newcommand{\fin}{\mathbf{\Upsilon}}
\DeclareMathOperator{\id}{id}
\DeclareMathOperator{\sk}{sk}

\usepackage{mathrsfs}
\newcommand{\edgemap}{\mathscr{E}}
\newcommand{\bous}{\mathscr{B}}

\theoremstyle{plain}
\newtheorem{theorem}{Theorem}
\newtheorem{lemma}[theorem]{Lemma}
\newtheorem{proposition}[theorem]{Proposition}
\newtheorem{corollary}[theorem]{Corollary}
\theoremstyle{definition}
\newtheorem{definition}[theorem]{Definition}
\newtheorem{example}[theorem]{Example}
\newtheorem{notation}[theorem]{Notation}
\theoremstyle{remark}
\newtheorem{remark}[theorem]{Remark}

\begin{document}

\title{Dimension and partial groups}
\author{Philip Hackney}
\address{Department of Mathematics, University of Louisiana at Lafayette, USA}
\email{philip@phck.net}
\thanks{The first author was supported by the Louisiana Board of Regents through the Board of Regents Support fund, contract number LEQSF(2024-27)-RD-A-31. 
This work was supported by a grant from the Simons Foundation (PH \#850849). 
The second author was partially supported by IRN MaDeF (CNRS) and MIAI@Grenoble Alpes (ANR-19-P3IA-0003).
}

\author{Rémi Molinier}
 \address{Institut Fourier, UMR 5582, Laboratoire de Mathématiques, Université Grenoble Alpes, CS 40700, 38058 Grenoble cedex 9, France}
\email{remi.molinier@univ-grenoble-alpes.fr} 
% ORCID: \href{https://orcid.org/0000-0002-3742-5307}{0000-0002-3742-5307}


 % Simplicial sets and complexes in algebraic topology

\begin{abstract}
A partial group with $n+1$ elements is, when regarded as a symmetric simplicial set, of dimension at most $n$.
This dimension is $n$ if and only if the partial group is a group.
As a consequence of the first statement, finite partial groups are genuinely finite, despite being seemingly specified by infinitely much data. 
In particular, finite partial groups have only finitely many im-partial subgroups.
We also consider dimension of partial groupoids.
\end{abstract}

\maketitle

If $G$ is a nontrivial finite group, then the classifying space of $G$ is infinite dimensional.
Considering the classifying space as a simplicial set in the usual way, with set of $m$-simplices equal to $G^m$, there are nondegenerate simplices in arbitrarily high degree.
However, this simplicial set admits additional structure, namely that of a \emph{symmetric (simplicial) set}.
Considered as a symmetric set rather than a simplicial set, the classifying space of a group with $n+1$ elements turns out to have dimension $n$ (\cref{pg n skel}).
The basic reason for this striking fact is that the symmetric indexing category has far more codegeneracy maps than the simplicial category has, so it is much easier for an element to be degenerate, and therefore much easier for a symmetric set to be $n$-skeletal than for its underlying simplicial set to be so.
As an example, a 2-simplex of the form $(g^{-1},g) \in G^2$ is degenerate in the simplicial sense only when $g$ is the identity, but it is always degenerate in the symmetric sense (\cref{ex g inverse g}).

This paper is primarily concerned with this symmetric dimension in the case of partial groups. 
These partial groups arose in Chermak's proof of the existence and uniqueness of centric linking systems for saturated fusions systems \cite{Chermak:FSL}.
We will show that a partial group with $n+1$ elements has dimension at most $n$, and has dimension $n$ if and only if it is a group (\cref{pg n skel,cor non-idens dimension}).
In Chermak's definition, a nontrivial partial group is always specified by infinitely much data.
Beyond an underlying set $\mathcal{M}$, one must provide a necessarily infinite subset of the free monoid on $\mathcal{M}$, consisting of those words which may be multiplied.
We demonstrate that this apparent infiniteness is illusory, as finite dimension of finite partial groups allows us to consider them as finite objects.

\section{Background}
We first provide a brief reminder of the symmetric sets perspective on partial groups and partial groupoids due to the first author and Lynd \cite{HackneyLynd}.
The reader is referred there for more detail, to \cite[Definition 2.1]{Chermak:FSL} for the original definition of partial group, and to \cite{Gonzalez:ETPGL} for a simplicial sets perspective.
Following this, we introduce the relevant notions of skeleta and dimension for symmetric sets, which naturally arise due to the fact that the category $\fin$ below is an `Eilenberg--Zilber category' (see \cite{Campion:CSEZC} for a good overview of general theory).

\subsection{Partial groups as symmetric sets}\label{sec partial group symmetric sets}
We let $\fin$ denote the category whose objects are the sets $[n] = \{0, 1, \dots, n\}$ for $n\geq 0$ and whose morphisms are arbitrary functions.
Write $\sym \coloneqq \fun(\fin^\op, \set)$ for the category of \emph{symmetric (simplicial) sets}. 
For a symmetric set $X$, we write $X_n$ for the value at the object $[n]$.
For $\alpha \colon [m] \to [n]$, we write the associated map $X_n \to X_m$ as $x\mapsto x \cdot \alpha$.
When $\alpha \colon [1] \to [n]$ has $\alpha(0) = i$ and $\alpha(1) = j$, we will use the special notation $x_{ij}$ for $x \cdot \alpha \in X_1$.
It is also convenient to use categorical conventions in low degrees:

\begin{notation}\label{not cat conv}
The following terminology, indicated by \emph{italics}, introduces these conventions:
We call elements of $X_0$ \emph{objects} and elements of $X_1$ edges or \emph{morphisms}.
We write $f\colon a \to b$ for an edge $f\in X_1$ having \emph{domain} $f \cdot \iota_0 = a$ and \emph{codomain} $f\cdot \iota_1 = b$, where $\iota_k \colon [0] \to [1]$ is $0\mapsto k$. 
The unique map $[1] \to [0]$ gives an injection $X_0 \to X_1$ sending an object $x$ to the \emph{identity} $\id_x$.
The \emph{inverse} $f^{-1} \colon b \to a$ of $f\colon a\to b$ is $f_{10} = f\cdot \tau$, where $\tau$ is the nontrivial automorphism of $[1]$. 
If $g\colon b\to c$ is another edge and there is a unique simplex $x\in X_2$ with $x_{01} = f$ and $x_{12} = g$, we will write $gf = x_{02}$ for the \emph{composite}.
\end{notation}

\begin{definition}
For $n\geq 1$, a \emph{spine} of $[n]$ is a tree having $[n]$ as its set of vertices. 
\end{definition}
We can regard a spine $T$ as a subposet of the powerset of $[n]$ consisting of those two-element subsets which constitute edges along with all of the one-element subsets.
For a symmetric set $X$, we stretch notation and write $\lim_T X$ for the limit of the induced diagram of sets.
Namely, there is a functor $T \to \fin$ sending $\{i,j\}$ to $[1]$, $\{i\}$ to $[0]$, and the inclusion $\{i\} \to \{i,j\}$ to $\iota_0$ if $i < j$, and to $\iota_1$ if $i > j$, and $\lim_T X$ is the limit of the composite functor $T^\op \to \fin^\op \to \set$.
The map $\lim_T X \to \prod_{ \{i, j\} \in T } X_1$ is injective since each $i$ is contained in at least one edge $\{i,j\}$ of the spine $T$. 
We then have a map
\[
X_n \xrightarrow{\quad\edgemap_T\quad} \lim_T X \xhookrightarrow{\quad\quad} \prod_{ \{i, j\} \in T } X_1
\]
where the $\{i, j\}$-component of $\edgemap_T(x)$ is $x_{ij}$ if $i<j$ and $\{i,j\} \in T$. (In \cite{HackneyLynd}, the map $\edgemap_T$ is taken to have codomain the product, rather than the limit.)

We are particularly interested in two spines on $[n]$ -- the \emph{standard} one whose edges are of the form $\{i-1,i\}$, and the \emph{starry} one whose edges are of the form $\{0,i\}$.
As trees these are not isomorphic for $n \geq 3$.
In both cases, the limit takes the form $\lim_T X = X_1 \times_{X_0} X_1 \cdots \times_{X_0} X_1$, but the fiber products use different maps $X_1 \to X_0$.
When $T$ is the standard spine on $[n]$, we write $\edgemap_n \colon X_n \to X_1 \times_{X_0} X_1 \cdots \times_{X_0} X_1$ for the usual Segal map $x\mapsto (x_{01}, x_{12}, \dots, x_{(n{-}1)n})$ and when $T$ is the starry spine we write $\bous_n \colon X_n \to X_1 \times_{X_0} X_1 \cdots \times_{X_0} X_1$ for the Bousfield--Segal map $x\mapsto (x_{01}, x_{02}, \dots, x_{0n})$ (see \cite[\S6]{Bergner:AID2}).

\begin{definition}
A symmetric set $X$ is \emph{spiny} if for every $n\geq 1$ and every spine $T$ of $[n]$, the map $\edgemap_T$ is injective.
\end{definition}

\begin{theorem}\label{groupoid characterizations}
Let $X$ be a symmetric set.
The symmetric set $X$ is spiny if and only if for each $n \geq 1$ there is a spine $T$ of $[n]$ such that $\edgemap_T$ is injective.
The following are equivalent:
\begin{enumerate}
\item For each $n\geq 1$, the map $\edgemap_T$ is a bijection for every spine $T$ of $[n]$.\label{item heart all} 
\item For each $n \geq 1$ there is a spine $T$ of $[n]$ such that $\edgemap_T$ is bijective.\label{item heart some}
\item The symmetric set $X$ is isomorphic to the nerve of a groupoid.\label{item heart special}
\end{enumerate}
\end{theorem}
\begin{proof}
The first statement is \cite[Theorem 3.6]{HackneyLynd}, and its proof is readily adapted to establish the equivalence of \eqref{item heart all} and \eqref{item heart some}.
A consequence of \cite[Proposition 4.1]{Grothendieck:TCTGA3} is that $X$ is a nerve of a groupoid if and only if the Segal map $\edgemap_n$ is a bijection for each $n\geq 1$.
Thus \eqref{item heart special} implies \eqref{item heart some} and is implied by \eqref{item heart all}.
\end{proof}

A symmetric set $X$ is said to be \emph{reduced} if $X_0$ has only one element.
In \cite{HackneyLynd} it was shown that the category of partial groups is equivalent to the category of reduced spiny symmetric sets.
Such a symmetric set $X$ gives, in the original definition of \cite{Chermak:FSL}, a partial group structure on $X_1$. 
We henceforth will refer to reduced spiny symmetric sets as \emph{partial groups}, and will refer to spiny symmetric sets as \emph{partial groupoids}.
If $X$ is a partial groupoid, then there is an injective function
\[
	X_n \to \operatorname{Mat}_{n+1,n+1}(X_1) \qquad x \mapsto (x_{ij}) = 
	\begin{bmatrix}
x_{00} & x_{01} & x_{02} & \cdots & x_{0n} \\
x_{10} & x_{11} & x_{12} & \cdots & x_{1n} \\
\vdots & & & & \vdots \\
x_{n0} & x_{n1} & x_{n2} & \cdots & x_{nn}
\end{bmatrix}
\]
taking values in $(n+1)$-by-$(n+1)$ matrices. 
We frequently identify $x \in X_n$ with its associated matrix $(x_{ij})$.
One can read off essential information from this matrix form, for instance $x_{ij}^{-1} = x_{ji}$ and $x_{jk} x_{ij} = x_{ik}$.
The original definition prioritizes the superdiagonal $\edgemap_n(x) = (x_{01}, x_{12}, \dots, x_{(n{-}1)n})$.

\subsection{Skeleta and dimension}
Since $\fin$ is an Eilenberg--Zilber category (see \cite[Examples 6.8]{BergerMoerdijk:OENRC}), symmetric sets have robust notions of dimension and well-behaved skeletal filtrations.
Let $X$ be a symmetric set.
An element $x \in X_n$ is \emph{degenerate} if there is a noninvertible surjection $\sigma \colon [n] \to [k]$ and an element $y\in X_k$ such that $x = y \cdot \sigma$.
Otherwise $x$ is said to be \emph{nondegenerate}.
Notice that being degenerate or nondegenerate is invariant under the action of automorphisms of $[n]$ on $X_n$.

\begin{definition}\label{def skeletal}
If $X$ is a nonempty symmetric set and $n\geq 0$, then the \emph{$n$-skeleton} $\sk_n X \subseteq X$ is the smallest symmetric subset containing all of the $n$-simplices of $X$.
We say that $X$ is \emph{$n$-skeletal} if $\sk_n X = X$.
If $X$ is $n$-skeletal but not $(n{-}1)$-skeletal, we say that $X$ has \emph{dimension} $n$.
\end{definition}

The $n$-skeleton of $X$ can also be obtained by first restricting along $\iota \colon \fin_{\leq n} \to \fin$ to the full subcategory containing the objects $[k]$ for $k\leq n$, and then left Kan extending along this same functor: $\sk_n X \cong \iota_! \iota^* X$ by \cite[Corollary 6.10]{BergerMoerdijk:OENRC}.
The elements of $\sk_n X$ of degree $m > n$ are all degenerate on elements in degree $n$,
meaning that they are all images of elements of $X_n$ under the action of surjective maps $[m] \to [n]$.

We warn that all of these concepts behave differently compared to their counterparts in simplicial sets.
It was already mentioned in the introduction that nontrivial finite groups are infinite dimensional as simplicial sets but finite dimensional as symmetric sets.
There is also a numerical difference with products:

\begin{remark}
Let $X$ and $Y$ be symmetric sets such that $X$ has dimension $n$ and $Y$ has dimension $m$. 
Then $X\times Y$ has dimension $q$ where $q=nm+n+m$.
This boils down to the fact that the set $[n] \times [m]$ has $(n+1)(m+1) = q+1$ elements (so $[n] \times [m] \cong [q]$).
In contrast, if $X$ and $Y$ are simplicial sets of dimensions $n$ and $m$, then $X\times Y$ has dimension $n+m$ \cite[\href{https://kerodon.net/tag/04ZS}{Tag 04ZS}]{kerodon}.
\end{remark}

\begin{lemma}\label{lem degenerate element}
Let $X$ be a partial groupoid, and $x = (x_{ij}) \in X_m$.
Then $x$ is degenerate if and only if $x_{ij}$ is an identity for some $i\neq j$.
\end{lemma}
\begin{proof}
Suppose $x = y \cdot \sigma$ for $\sigma \colon [m] \to [m-1]$ and $y\in X_{m-1}$.
Then there are $i\neq j$ and $k$ with $\sigma(i) = k = \sigma(j)$.
We then have $x_{ij} = y_{kk}$ is an identity.

Suppose $x_{ij}$ is an identity for some $i\neq j$, and let $\sigma \colon [m] \to [m-1]$ be any surjective function with $\sigma(i) = \sigma(j) = k$.
Let $\delta \colon [m-1] \to [m]$ be the section of $\sigma$ with $\delta (k) = i$, and let $y = (x \cdot \delta) \cdot \sigma \in X_m$.
Since $\delta \sigma (t) = t$ for $t\neq j$, we have $y_{0t} = x_{0t}$ so long as $t\neq j$.
But $y_{0j} = x_{0i} = x_{ij} x_{0i} = x_{0j}$. 
Since $X$ is spiny, $x$ is equal to $y$, hence is degenerate.
\end{proof}

\begin{example}\label{ex g inverse g}
Let $X$ be a partial groupoid and consider a 2-simplex $x \in X_2$ such that $\edgemap_2(x) = (f^{-1}, f)$ with $f\in X_1$. 
By \cref{lem degenerate element}, this 2-simplex is degenerate since $x_{02} = x_{12}x_{01} = ff^{-1}$ is the identity.
This can also be seen directly, as $x = f\cdot \chi$, where $\chi$ is the surjective map $[2] \to [1]$ sending $i$ to $|1-i|$.
\end{example}

\begin{comment}
\begin{remark}
Though the forgetful functor $\sym \to \sset$ does not preserve skeletality, it does preserve coskeletality.
This is a formal consequence of the left adjoint symmetrization functor $\sset \to \sym$ preserving $n$-skeleta (see \cite[Lemma 8.2.3]{Antokoletz:NAMCHT}).
A symmetric set $X$ being $n$-coskeletal means that is the right Kan extension of $\iota^* X$ along the inclusion $\iota \colon \fin_{\leq n} \to \fin$. 
\end{remark}
\end{comment}

\begin{theorem}\label{skel spiny thm}
Suppose $X$ is a $q$-skeletal symmetric set.
Then $X$ is spiny if and only if $\edgemap_T$ is injective for some (or every) spine $T$ of $[n]$ for each $1 \leq n \leq q$.
\end{theorem}

This theorem will not play a role in subsequent sections, and its proof appears as the appendix.
It follows from \cref{skel spiny thm} that 1-skeletal symmetric sets are automatically spiny.

\begin{remark}
The category of partial groups of dimension at most $n$ is equivalent to the (reflective) subcategory of $\fun(\fin_{\leq n}^\op, \set)$, consisting of those reduced objects which are spiny in the appropriate sense.
\end{remark}

\begin{example}
The remark implies there are exactly two indecomposable partial groups which are one-dimensional: the cyclic group of order 2, and the free partial group on one generator.
All other one-dimensional partial groups are coproducts of these.
By \emph{indecomposable} we mean that $X$ cannot be written nontrivially as a coproduct $A \vee B$ in the category of partial groups.
\end{example}

\section{Skeletality of partial groupoids}

A symmetric set is $n$-skeletal if and only if each of its connected components is $n$-skeletal, and if it is nonempty then it is $n$-dimensional if and only if it is $n$-skeletal and has an $n$-dimensional connected component.
Here, a symmetric set is called \emph{connected} if its underlying simplicial set is connected.
As such, in this section we mostly restrict attention to connected partial groupoids. 
We use the convention that connected objects are nonempty.

\begin{definition}\label{def size p}
Suppose $X$ is a partial groupoid with $X_1$ nonempty and finite.
By abuse of notation, if $x\in X_0$ we let $\hom(x,-)$ denote the set of edges $x\to y$ for some $y$.
For each $x\in X_0$, let $n_x$ be the cardinality of the set $\hom(x,-) \setminus \{\id_x\}$ of nonidentity edges $x\to y$, and let $p = p(X)$ be the maximum among all $n_x$.
\end{definition}
If $X$ is a partial group, then $p$ is just the number of  nonidentity elements in $X_1$.

\begin{proposition}\label{dimension upper bound}
If $X$ is a nonempty finite partial groupoid, then $X$ is $p$-skeletal.
\end{proposition}
\begin{proof}
It is enough to prove that each $x = (x_{ij}) \in X_m$ is degenerate for $m > p$.
The elements $x_{0a}$ all share a common domain $y \in X_0$.
Since the set $\{ x_{0a} \mid a \in [m] \} \subseteq X_1$ has at most $p+1$ elements, there is a pair $a < b$ with $x_{0a} = x_{0b}$.
We then have $x_{ab} = x_{0b}x_{a0} = x_{0b}x_{0a}^{-1}$ is an identity, so $x$ is degenerate.
\end{proof}

\begin{proposition}\label{prop finite conn groupoid}
If $X$ is a finite connected groupoid, then $X$ has dimension $p$.
\end{proposition}
\begin{proof}
As $X$ is a connected groupoid, the cardinality of $\hom(x,-)$ does not depend on the object $x$.
Since the Bousfield--Segal map $\bous_p$ is a bijection, there is an element $(f_{ij}) \in X_p$ with $\{ f_{01}, \dots, f_{0p} \}$ the set of nonidentity morphisms with domain $x$.
Then $f_{ij} = f_{0j} f_{i0} = f_{0j} f_{0i}^{-1}$ is an identity if and only if $i=j$.
Hence $(f_{ij}) \in X_p$ is nondegenerate, so $X$ is not $(p{-}1)$-skeletal.
\end{proof}

If $X$ is a finite connected groupoid and $x\in X_0$, then $p + 1 = |\hom(x,x)| \times |X_0|$. 

\begin{corollary}\label{pg n skel}
If $X$ is a partial group with $n+1$ elements, then $X$ is $n$-skeletal. 
If $X$ is actually a group, then $X$ has dimension $n$.
\qed
\end{corollary}

\begin{example}
If $X$ is a nonempty partial groupoid of dimension $n$, then its reduction $\mathcal{R}X$ also has dimension $n$ (see \cite[\S5.1]{HackneyLynd}).
This is easily seen since the canonical map $X \to \mathcal{R}X$ is surjective (which implies that $\mathcal{R}X$ is $n$-skeletal) and preserves nondegenerate simplices (which implies that $\mathcal{R}X$ has dimension at least $n$).
Notice that $p(\mathcal{R}X) > p(X)$ whenever $X$ is a connected finite partial groupoid with $|X_0| > 1$.
\end{example}

\begin{example}[Localities]
Suppose $X = (\mathcal{M}, \Delta)$ is a finite objective partial group \cite[Definition 2.6]{Chermak:FSL}, and let $Y$ be the maximal subgroupoid of its associated transporter category from \cite[Remark 2.8]{Chermak:FSL}.
The connected components of $Y$ correspond to conjugacy classes of $\Delta$.
The dimension of $Y$ is $\max_{P\in\Delta}(|P^\mathcal{M}|\times |N_\mathcal{M}(P)|)-1$ where $P^\mathcal{M}$ is the conjugacy class of $P$ in $\mathcal{M}$ and $N_\mathcal{M}(P)$ the normalizer of $P$ in $\mathcal{M}$ \cite[p.62]{Chermak:FSL}.
This is also the dimension of $X$, since the canonical surjective map $Y\to X$ preserves nondegenerate elements.
In particular, if $(\mathcal{L},\Delta,S)$ is a locality with fusion system $\mathcal{F}$ \cite[Definition 2.9]{Chermak:FSL}, then its dimension is $\max_{P\in\Delta}(|P^\mathcal{F}|\times |N_\mathcal{L}(P)|)-1$. 
Moreover, if $(\mathcal{L},\Delta,S)$ is a $\Delta$-linking system \cite[below Definition 2.9]{Chermak:FSL}, its dimension is $\max_{P\in\Delta}(|P^\mathcal{F}|\times|Z(P)|\times |\operatorname{Aut}_\mathcal{F}(P)|)-1$ which only depends on $\Delta$ and $\mathcal{F}$.
\end{example}

\begin{theorem}\label{dimension implies groupoid}
Suppose $X$ is a connected partial groupoid.
If $X$ has dimension $p$, then $X$ is a groupoid.
\end{theorem}
\begin{proof}
We wish to show that the Bousfield--Segal map $\bous_m$, which by hypothesis is an injection, is in fact a bijection for every $m \geq 1$.
Consider $(g_{01}, \dots, g_{0m}) \in \prod_{i=1}^m X_1$ with the $g_{0i}$ having common domain $y$. 
We show this element is in the image of $\bous_m$.

Since $X$ has dimension $p$, there is a nondegenerate $p$-simplex $(f_{ij})$.
For each $t$, let $x_t$ be the common domain of the $f_{tk}$ (so $f_{tt} = \id_{x_t}$). 
By \cref{lem degenerate element}, the set 
\[ \{ f_{t0}, f_{t1}, \cdots, f_{tp} \} \subseteq \hom(x_t, -) \subseteq  X_1 \] has cardinality $p+1$ (using $f_{ij} = f_{tj} f_{it} = f_{tj} f_{ti}^{-1}$ is an identity only for $i=j$).
But $\hom(x_t,-)$ has cardinality at most $p+1$, so we have $\{ f_{t0}, f_{t1}, \cdots, f_{tp} \} = \hom(x_t, -)$ for each $t$.
We observe that every object of $X$ appears among the $x_t$.
Indeed, if there is a morphism $f \colon x_t \to z$, then $f = f_{tk}$ for some $k$, hence $f^{-1} = f_{kt}$ has domain $z = x_k$.
So any object which is joined by an edge to a member of $\{x_0, \dots, x_p\}$ is itself a member of this set.
Since $X$ is connected, there is always a zig-zag of edges between any two objects, and we conclude inductively that $X_0 = \{ x_0, \dots, x_p \}$.

Suppose $y=x_t$, and write $\zeta \colon \hom(y, -) \to \{ 0, 1, \dots, p \}$ for the bijection with $\zeta(f_{ti}) = i$.
Define a function $\alpha \colon [m] \to [p]$ by $\alpha(0) = \zeta(f_{tt}) = t$ and $\alpha(i) = \zeta(g_{0i})$.
Then $(h_{ij}) \coloneq (f_{ij})\cdot \alpha = (f_{\alpha(i),\alpha(j)})$ has for $1 \leq j \leq m$
\[
 	\zeta(h_{0j}) = \zeta(f_{t,\alpha(j)}) = \alpha(j) = \zeta(g_{0j})
\] 
so $h_{0j} = g_{0j}$.
Thus $\bous_m ((f_{ij}) \cdot \alpha) = (g_{01}, \dots, g_{0m})$, as desired.
\end{proof}

The map $X_0 \to X_1$ induced by the unique map $[1] \to [0]$ is injective, and we will use the notation $X_1 \setminus X_0$ as shorthand for the set of nonidentity edges of $X$.
The following corollary applies in particular when the dimension of $X$ is $|X_1| - 1$.

\begin{corollary}\label{cor non-idens dimension}
Let $X$ be a connected partial groupoid. If $X$ has dimension $m \geq |X_1\setminus X_0|$, then $X$ is a group.
\end{corollary}
\begin{proof}
Since $\hom(x,-) \setminus \{\id_x\} \subseteq X_1 \setminus X_0$ for each $x$, then we have $n_x \leq |X_1 \setminus X_0|$ (using the notation of \cref{def size p}), hence $p\leq |X_1 \setminus X_0| \leq m$.
The dimension of $X$ is at most $p$ by \cref{dimension upper bound}, so $p=m$.
\Cref{dimension implies groupoid} implies that $X$ is a groupoid.
We then notice that $X$ has exactly one object.
Indeed, there is an $x$ with $\hom(x,-) \setminus \{\id_x\} = X_1 \setminus X_0$.
If $y$ is an object and $f$ is a nonidentity morphism with domain or codomain $y$, then both $f$ and $f^{-1}$ are in $\hom(x,-)$, hence $y=x$.
\end{proof}

\section{Finiteness}

In light of \cref{pg n skel}, if $X$ is a finite partial group, (meaning just that $X_1$ is finite), then $X$ may \emph{truly} be regarded as a finite object.
For example, we have the following proposition about im-partial subgroups (in the sense of \cite[\S 9]{ChermakFL3}).
An im-partial subgroup of $X$ is nothing but a nonempty symmetric subset $Y\subseteq X$.

\begin{corollary}\label{prop finite impartial}
If $X$ is a finite partial group, then $X$ has only finitely many im-partial subgroups.
\end{corollary}

This is an immediate consequence of the following:

\begin{proposition}
There are, up to isomorphism, only finitely many partial groups with a given finite cardinality.
\end{proposition}
\begin{proof}
Let $X_1$ have cardinality $n+1$.
The subcategory $\fin_{\leq n} \subset \fin$ is finite, so there are only finitely many functors $\fin_{\leq n}^\op \to \set$ sending $[k]$ to a subset of $X_1^k$.
\end{proof}

\appendix

\section{Proof of Theorem~\ref{skel spiny thm}}

We say a symmetric set is \emph{spiny in degrees below $q$} if $\edgemap_T$ is injective for every spine $T$ of $[n]$ for each $1 \leq n \leq q$. 
The inductive proof of Theorem 3.6 of \cite{HackneyLynd} shows that we can replace `every' by `some' in the preceding sentence.
The forward direction of the theorem is immediate.
Recall the Eilenberg--Zilber lemma (see \cite[Proposition 6.9]{BergerMoerdijk:OENRC}): 
\begin{lemma}
Let $X$ be a symmetric set.
If $x \in X_m$, then there is a pair $(y, \sigma)$ with $y \in X_k$ nondegenerate, $\sigma \colon [m] \twoheadrightarrow [k]$ surjective, and $x = y\cdot \sigma$. 
This data is essentially unique in the following sense: given another such pair $(y' \in X_{k'}, \sigma' \colon [m] \twoheadrightarrow [k'])$, there is a unique automorphism $\alpha$ with $y = y' \cdot \alpha$ and $\alpha \sigma = \sigma'$.
\end{lemma}

\begin{lemma}\label{lem sigma equals tau}
Suppose that $X$ is spiny in degrees below $q$.
Let $\sigma, \tau \colon [m] \twoheadrightarrow [k]$ be two surjective maps with $\sigma(0) = 0 = \tau(0)$ and $k\leq q$, and let $x \in X_k$ be nondegenerate.
If $\bous_m(x\cdot \sigma) = \bous_m(x\cdot \tau)$, then $\sigma = \tau$.
\end{lemma}
\begin{proof}
For all $i$ we have $x_{0,\sigma(i)} = (x\cdot \sigma)_{0i} = (x\cdot \tau)_{0i} = x_{0,\tau(i)}$.
The existence of $i$ with $\sigma(i) \neq \tau(i)$ implies $x$ is degenerate (as in the proof of \cref{lem degenerate element}).
\end{proof}

\begin{lemma}\label{lem endo iso}
Suppose $x\in X_k$ is nondegenerate and $x = x\cdot \alpha$ for some endomorphism $\alpha \colon [k] \to [k]$.
Then $\alpha$ is an isomorphism.
\end{lemma}
\begin{proof}
Factor $\alpha$ as $\alpha^+\alpha^-$ with $\alpha^-$ surjective and $\alpha^+$ injective.
Using the Eilenberg--Zilber lemma, we write $x\cdot \alpha^+$ as $z \cdot \sigma$ for some surjective map $\sigma$ and some nondegenerate element $z$.
Thus $x = x\cdot \alpha = z\cdot (\sigma \alpha^-)$, so by the uniqueness part of the Eilenberg--Zilber lemma, $\sigma \alpha^-$ is an isomorphism.
In particular, $\alpha^-$ is an automorphism of $[k]$, from which we deduce that $\alpha$ is an isomorphism since it is an injective endomorphism of the finite set $[k]$.
\end{proof}

\begin{lemma}\label{lem different elts}
Suppose $X$ is spiny in degrees below $q$.
Let $k,\ell \leq q$ and  $x \in X_k$ and $y\in X_\ell$ are nondegenerate.
If $\sigma \colon [m] \twoheadrightarrow [k]$ and $\tau \colon [m] \twoheadrightarrow [\ell]$ are surjective maps with $\sigma(0) = 0 = \tau(0)$ and $\bous_m(x\cdot \sigma) = \bous_m(y\cdot \tau)$, then $x\cdot \sigma = y \cdot \tau$.
\end{lemma}
\begin{proof}
Let $\delta \colon [k] \rightarrowtail [m]$ be a section of $\sigma$ with $\delta(0) = 0$ and $\epsilon \colon [\ell] \rightarrowtail [m]$ be a section of $\tau$ with $\epsilon(0) = 0$.
Since $\bous_m(x\cdot \sigma) = \bous_m(y\cdot \tau)$, for each $i$ we have $(x\cdot \sigma)_{0i} = x_{0,\sigma(i)}$ equal to $(y\cdot \tau)_{0i} = y_{0,\tau(i)}$.
Notice that \[ (y \cdot \tau \delta)_{0j} = y_{0, \tau \delta(j)}  = (y\cdot \tau)_{0,\delta(j)} = (x\cdot \sigma)_{0,\delta(j)} = x_{0,\sigma \delta(j)} = x_{0j}. \] 
Since $\bous_k (y\cdot \tau \delta) = \bous_k (x)$ and $k \leq q$, we have $x = y \cdot \tau \delta$.
A similar proof gives $y = x \cdot (\sigma \epsilon)$.

We thus have $x = x \cdot (\tau \delta \sigma \epsilon)$, so combining \cref{lem sigma equals tau} and \cref{lem endo iso} we see that $\tau \delta \sigma \epsilon$ is the identity on $[k]$. 
Likewise, $\sigma \epsilon \tau \delta$ is the identity on $[\ell]$, so $\ell = k$ and the maps $\tau \delta$, $\sigma \epsilon$ are inverse isomorphisms.
In particular, $\sigma \epsilon \tau$ is surjective, and
$y \cdot \tau = (x \cdot (\sigma \epsilon) ) \cdot \tau = x \cdot (\sigma \epsilon \tau)$.
By assumption, $\bous_m ( x\cdot (\sigma \epsilon \tau) ) = \bous_m (y \cdot \tau) = \bous_m (x \cdot \sigma)$, so \cref{lem sigma equals tau} implies that $\sigma \epsilon \tau = \sigma$.
We thus have $x\cdot \sigma = x \cdot (\sigma \epsilon \tau) = y \cdot \tau$.
\end{proof}

To prove \cref{skel spiny thm}, we just need to show that we can reduce to the situation of \cref{lem different elts}.
Given an element $z \in X_m$ with $m > q$, we can find a pair $(w, \gamma)$ with $w\in X_k$ nondegenerate, $\gamma \colon [m] \twoheadrightarrow [k]$ surjective, and $z = w \cdot \gamma$.
In particular, $k \leq q$.
If $\gamma(0) = i > 0$, we let $\alpha \colon [k] \to [k]$ be the bijection which interchanges $0$ and $i$ and fixes all other points.
Then $(w\cdot \alpha) \cdot \alpha \gamma = w \cdot \gamma = z$, so $(w\cdot \alpha, \alpha \gamma)$ is a pair as in the Eilenberg--Zilber lemma, and of course $\alpha \gamma(0) = \alpha(i) = 0$.
So we can always choose a representative of $z$ of the form $(x, \sigma)$ where $\sigma(0) = 0$.
We may thus conclude from \cref{lem different elts} that $\bous_m$ is injective.

\subsection*{Acknowledgements}
This work originated when the authors were visiting the Copenhagen Centre for Geometry and Topology, and we are grateful for its hospitality.
We also extend thanks to Jan Steinebrunner for a question that sparked this investigation, and to Justin Lynd for valuable conversations about this work.
We thank the referee for helpful feedback that improved the paper.

\bibliographystyle{plain}
\bibliography{skel}

\end{document}